\documentclass[11pt]{article}   
\usepackage{amsthm, amsmath,amssymb}   
\theoremstyle{plain}
\newtheorem{theorem}{Theorem}[section]

\theoremstyle{definition}
\newtheorem{definition}[theorem]{Definition}
\newtheorem{example}[theorem]{Example}

\theoremstyle{remark}
\newtheorem{remark}[theorem]{Remark}

\title{Constructive proof of extended Kapranov theorem}
\author{Luis Felipe Tabera\footnote{Supported by the European Research Training Network RAAG (HPRN-CT-2001-00271), a FPU research grant and the project MTM2005-08690-C02-02 from the Spanish Ministerio de Educaci\'on y Ciencia.}\phantom{x}\footnote{First publiched in the Proceedings of the ``X Encuentro en \'Algebra Computacional y Aplicaciones", EACA 2006, Sevilla, (F. J. Castro Jimenez, J. M. Ucha Enriquez editors), ISBN: 84-611-2311-5, pp 178-181.}}
\date{\today}

\begin{document}
\maketitle

\begin{abstract}
Kapranov Theorem is a well known generalization of Newton-Puiseux theorem for the case of several variables.
This theorem is stated mainly in the context of tropical geometry.
We present a new, constructive proof, that also characterizes the possible principal terms of points in a hypersurface contained in the algebraic torus $(\mathbb{K}^*)^n$.
\end{abstract}

\section*{Introduction}

Newton-Puiseux method for computing Puiseux series \cite{wal50} is a fundamental tool in the study of singularities of algebraic curves.
It proves that the field $\mathbb{K}$ of Puiseux series is an algebraically closed field.
But, in addition, it gives the possible order of the roots of a polynomial using the so called Newton polygon method.
More precisely, let us write every element $\widetilde{a}\in\mathbb{K}$ as a fractional power series over $\mathbb{C}$,
\[\widetilde{a}=\sum_{i= k}^{\infty}\alpha_it^{\frac{i}{q}},\ k\in\mathbb{Z},\ q \in \mathbb{N}^*,\ \alpha_i \in \mathbb{C}.\]
The order of a non zero series is $o(\widetilde{a})=\min\{i\ |\ \alpha_i\neq 0\} \in \mathbb{Q}$, $o(0)=\infty$.
This order is a valuation on $\mathbb{K}$, with valuation group $\mathbb{Q}$ and residue field $\mathbb{C}$.
The principal coefficient of an element $\widetilde{a}\in\mathbb{K}$ is $pc(\widetilde{a})=\alpha_{o(\widetilde{a})}\in \mathbb{C}$.
An element will usually be written as $\widetilde{a}=\alpha_{o(\widetilde{a})} t^{o(\widetilde{a})} + O(t^{ o( \widetilde{a}) +\epsilon})$ to emphasize the principal term of a series.
In the Newton-Puiseux process, it is shown that the orders of the (non-zero) roots of an univariate polynomial $\widetilde{f}(x)=\sum_{i=0}^m \widetilde{a}_ix^i$ are precisely the numbers $r\in\mathbb{Q}$ such that the value of $f(r)=\min_i\{o(\widetilde{a}_i )+ ir\}$ is attained for at least two different indices $i$.

This is the aspect generalized by Kapranov theorem when the polynomial is multivariate.
Given $\widetilde{f}=\sum_{i\in I} \widetilde{a}_ix^i\in\mathbb{K}[x_1,\ldots,x_n]$, $x^i=x_1^{i_1}\ldots x_n^{i_n}$, consider the piecewise affine map $f(x)=\min_{i\in I}\{o(\widetilde{a}_i) +ix\}$, where $ix$ denotes the standard scalar product $i_1x_1+\ldots +i_nx_n$.
Consider $V_{\widetilde{f}}$ the ``non zero'' roots of $\widetilde{f}$.
That is, the points in the variety defined by $\widetilde{f}$ such that no coordinate is zero, $V_{\widetilde{f}}\subseteq (\mathbb{K}^*)^n$.
Kapranov theorem (\emph{cf.} \cite{ekl00}) proves that the image by the order map applied componentwise $O:(\mathbb{K}^*)^n\longrightarrow \mathbb{Q}^n$ of $V_{\widetilde{f}}$  is the set of points
\[O(V_{\widetilde{f}})=\{r=(r_1,\ldots,r_n)\ |\ o(\widetilde{a}_i)+ir=o(\widetilde{a}_j)+jr\leq o(\widetilde{a}_k)+kr,\ k\neq i,j\}=\mathcal{T}(f).\]
That is, the points $(r_1,\ldots,r_n)$ such that the value $f(r_1,\ldots,r_n)$ is attained for at least two different indices $i$.
This set is denoted by $\mathcal{T}(f)$.

This theorem is essential in the context of tropical geometry.
Tropical geometry treats the sets $O(V_{\widetilde{f}})\subseteq\mathbb{Q}^n$ as geometric entities.
In many cases, these geometric entities (tropical varieties) are manipulated in order to derive useful information for the algebraic varieties $V_{\widetilde{f}}$.
Kapranov theorem means that, at least for the case of hypersurfaces, tropical varieties $O(V_{\widetilde{f}})$ can be defined algebraically via the piecewise affine functions $f(x_1,\ldots,x_n)$.
These functions are polynomials in the context of tropical semirings.

There are already several constructive proofs of Kapranov theorem, \cite{ekl00}, \cite{shu05}.
But they do not describe what the principal coefficients of the roots are.
This information about the principal coefficients is required in the study of geometric constructions presented in \cite{tab05}, \cite{tab06}.
So, it is needed an extended version of this theorem that also considers the principal coefficients and not just the orders.
Such a version appears in \cite{sp05}, not only for the case of hypersurfaces discussed here, but for general varieties.
The problem with this version is that it is not constructive, and we are looking (\emph{cf.} \cite{tab05}, \cite{tab06}) for effective ways of computing a root in $V_{\widetilde{f}}$ from its principal terms.
That is why we built up a new proof of the theorem inspired in the classical univariate one.

A last remark for the reader with knowledge of tropical geometry.
Tropical geometry usually deals with a characteristic zero algebraically closed valued field with characteristic zero residue field.
However, the presentation here is done exclusively over the field of Puiseux series, as it is easier to work with it.
Anyway, as the theory of this class of fields is complete \cite{rob56}, it is easy to check that this theorem can be rewritten in the first order language of the theory, and hence, it remains true in the whole class of fields.
This model theoretic aspect will not be developed here, though.

\section{Computing the preimage of a point in a hypersurface}

\begin{definition}
Let $\widetilde{f}=\sum_{i\in I} \widetilde{a}_ix^i \in\mathbb{K}[x]$ be a polynomial, $x=x_1,\ldots,x_n$, $i=i_1,\ldots,i_n$, $pc(\widetilde{a}_i)= \alpha_i$, $o(\widetilde{a}_i)=a_i$, $f(x)=\min_{i \in I}\{ a_i +ix\}$.
Let $b=(b_1,\ldots,b_n)\in \mathbb{Q}^n$.
Define \[\widetilde{f}_{b}(x_1,\ldots,x_n)=\sum_{\substack{i\in I\\ a_i+i_1 b_1\cdots + i_n b_n= f(b_1,\ldots,b_n)} } \alpha_ix^i=pc(\widetilde{f}(x_1t^{b_1},\ldots, x_n t^{b_n})),\] a polynomial over the residue field $\mathbb{C}$.
That is, rewrite the polynomial $\widetilde{f}(x_1t^{b_1},\ldots,x_nt^{b_n})$ as a series in $\mathbb{C}[x_1,\ldots,x_n]((t^{\frac{1}{q}}))$. For some $\epsilon>0$
\[\widetilde{f}(xt^{b})=\widetilde{f}_b(x)t^{f(b)}+O(t^{f(b)+\epsilon}).\]
\end{definition}

\begin{remark}\label{sobra}
Note that, by construction, the monomials of $\widetilde{f}_b$ correspond with the indices $i$ where $f(b)$ is attained.
It is immediate that the following sentences are equivalent:
\begin{itemize}
\item $b\in\mathcal{T}(f)$ ($b\in\mathbb{Q}^n$ by hypothesis).
\item $\widetilde{f}_b$ has at least two monomials.
\item $\widetilde{f}_b$ has a root in $(\mathbb{C}^*)^n.$
\end{itemize}
\end{remark}

\begin{theorem}[Newton-Puiseux lifting method in several variables, constructive version of extended Kapranov theorem]\label{initialserie}
Let $\widetilde{f}$ $=$ $\sum_{i\in I}\widetilde{a}_ix^i$ $\in \mathbb{K}[x_1,\ldots,x_n]$ be a polynomial.
Then, given $(\gamma_1, \ldots,\gamma_n)\in (\mathbb{C}^*)^n$, $(b_1,\ldots,b_n)\in \mathbb{Q}^n$, there is a point $\widetilde{b}=(\widetilde{b}_1,\ldots,\widetilde{b}_n)\in V(\widetilde{f})$ with $pc(\widetilde{b}_j)=\gamma_j, o(\widetilde{b}_j)=b_j$ if and only if $(b_1,\ldots,b_n)\in \mathcal{T}(f)$ and $\widetilde{f}_{b}(\gamma_1, \ldots,\gamma_n)=0$.
\end{theorem}

\begin{proof}
The only if part is trivial.
The if implication is proved by induction in $n$, being true for $n=1$ by the classical Newton-Puiseux theorem \cite{wal50}.
First of all, if $\widetilde{f}(\gamma t^{b})=0$, the root is trivially achieved.
second, if one variable $x_j$ does not appear in $\widetilde{f}_b$, it may be substituted by $x_j=\gamma_jt^{b_j}$ without substituting the hypotheses.
Thus, without loss of generality, suppose that the variables appearing in $\widetilde{f}_b$ are exactly $x_1,\ldots,x_n$.
In this situation, there are two possible cases:\\

$\bullet$ If there is a $j$, $1\leq j\leq n$, such that $\widetilde{f}_b(x_1,\ldots, \gamma_j,\ldots,x_n)\neq 0$ then, after reordering the variables if necessary, suppose that $j=1$.
Write $b=(b_1,b')$, $x=(x_1,x')$, $\gamma=(\gamma_1,\gamma')$.
The conditions needed in order to apply induction over $\widetilde{g}(x')=\widetilde{f}(\gamma_1 t^{b_1},x')$ are:
\[b'=(b_2,\ldots,b_n)\in\mathcal{T}(g);\quad \widetilde{g}_{b'}(\gamma_2,\ldots,\gamma_n)=0.\]
It is possible that $g \neq f(b_1,x)$, see example \ref{perdidamonomiosinduccion}.
But, as $\widetilde{g}(x')=\widetilde{f}(\gamma_1t^{b_1},x')$, we trivially verify that
\[\widetilde{g}(x't^{b'})=\widetilde{f}(\gamma_1t^{b_1},x't^{b'})= \widetilde{f}_b(\gamma_1,x')t^{f(b)}+O(t^{f(b)+\epsilon}).\]
So, $\widetilde{g}_{(b')}(\gamma')=\widetilde{f}_b(\gamma_1,\gamma')=0$ and the second condition holds.
By the equivalence given in \ref{sobra}, $\gamma_i\in \mathbb{C}^*$ implies that $(b_2,\ldots,b_n)\in \mathcal{T}(g)$.

$\bullet$ Suppose now that, for every $1\leq i\leq n$, $\widetilde{f}_b(x_1,\ldots, \gamma_i, \ldots,x_n)=0$.
In order to follow induction, recall $\widetilde{f}(xt^{b})=\widetilde{f}_bt^{f(b)}+O(t^{f(b)+\epsilon})$.
Write \[\widetilde{f}_b=(x_1-\gamma_1)^k(x_2-\gamma_2)\cdots(x_n-\gamma_n)q(x_1, \ldots,x_n); q(\gamma_1,x')\neq 0.\]
Substituting here $x_1$ by $\gamma_1$ as before would destroy the desired structure for the induction, so substitute $x_1$ by $\gamma_1+t^{\frac{\epsilon}{2k}} $ instead.
\[\widetilde{g}(x't^{b'})=\widetilde{f}((\gamma_1+t^{\frac{\epsilon}{2k}}) t^{b_1},x't^{b'})=t^{f(b)}\widetilde{f_b}(\gamma_1+t^{\frac{\epsilon}{2k}},x')+ O(t^{f(b)+\epsilon})=\]
\[=t^{f(b)+\frac{\epsilon}{2}}(x_2-\gamma_2)\cdots(x_n-\gamma_n)q(\gamma_1+ t^{\frac{\epsilon}{2k}},x_2,\ldots,x_n)+O(t^{f(b)+\epsilon}).\]
$q$ is a polynomial with coefficients in $\mathbb{C}$, so $q(\gamma_1+t^{\frac{\epsilon}{2k}},x')=q(\gamma_1,x')+O(t^\frac{\epsilon}{2k})$.
The previous expression equals:
\[t^{f(b)+\frac{\epsilon}{2}}(x_2-\gamma_2)\cdots(x_n-\gamma_n) q(\gamma_1,x_2, \ldots,x_n)+O(t^{f(b)+\frac{\epsilon}{2}+\frac{\epsilon}{2k}}).\]
Write $\widetilde{g}(x')=\widetilde{f}((\gamma_1+t^\frac{\epsilon}{2k}) t^{b_1},x')$.
The computations above yields,
$\widetilde{g}_{(b_2,\ldots,b_n)}(\gamma_2,\ldots,\gamma_n)=0$
and hence $(b_2,\ldots,b_n)\in\mathcal{T}(g)$.
Go on with the next induction step.
\end{proof}

Finally, the following example shows how the method works.

\begin{example}\label{perdidamonomiosinduccion}

Consider the polynomial
\[\widetilde{f}=-3t^2+3tx-t^2y+txy-t^3xy^4+(t^4+t^5)y^4+x^5,\]
$f=\min\{2,1+x,2+y,1+x+y,3+x+4y,4+4y,0+5x\}$.
Take $b=(1,0) \in \mathcal{T}(f)$, $\widetilde{f}(tx,y)=(-3+3x-y+xy)t^2+O(t^4)$ and hence $\widetilde{f}_b=-3+3x-y+xy$,
$\widetilde{f}_b(1,-3)=0$. So, by Theorem \ref{initialserie}, there is a root in $(\mathbb{K}^*)^2$ whose principal terms is $(t,-3)$.

As $\widetilde{f}_b(1,y)=\widetilde{f}_b(x,-3)=0$, we are in the second case of the theorem, so we make the substitution $x=t+t^2$ in $\widetilde{f}$.
\[\widetilde{f}(t+t^2,y)=\widetilde{g}(y)=3t^3+t^5+5t^6+10t^7+10t^8+5t^9+t^{10}+t^3y,\]
$g(y)=\min\{3,3+y\}$.
Note that $g(y)\neq f(1,y)=\min\{2,2+y,4+4y\}$ but, as claimed in the theorem, $0\in\mathcal{T}(g)$.
Now we use Newton-Puiseux method to compute one root of $\widetilde{g}(y)$ whose principal term is $-3$, the point is:
\[(x,y)\simeq(t+t^2,-3-t^2-5t^3-10t^4-10t^5-5t^6-t^7)\]
\end{example}

\hfill\begin{minipage}{3.5in}
{Luis Felipe Tabera\\
Dpto. Matem\'aticas, Estad\'\i{}stica y Computaci\'on,\\
F. Ciencias, U. Cantabria, 39071, Santander, Spain\\
\emph{e-mail}: luisfelipe.tabera@unican.es\\
\emph{url}: http://personales.unican.es/taberalf\\}
\end{minipage}
\end{document}